\documentclass[10pt]{amsart}
\usepackage{mathtools,amsthm,amsfonts,amssymb,amscd,amsbsy,amsmath,graphicx,bbm,tikz}
\usetikzlibrary{decorations.pathreplacing,}

\usepackage[margin=1in]{geometry}

\def\mcb{\mathcal B}

\def\mcch{\mathcal{ CH}}

\def\N{\mathbb N}

\def\sd{\triangle}

\def\a{\alpha}

\def\zb{\hat{\textbf{0}}}

\newcommand{\inv}{^{-1}}

\usepackage{scrextend}

\newtheorem{theorem}{Theorem}[section]
\newtheorem{conjecture}{Conjecture}[section]
\newtheorem{lemma}{Lemma}[section]
\newtheorem{prop}{Proposition}[section]
\newtheorem{cor}{Corollary}[section]
\newtheorem{rem}{Remark}[section]

\theoremstyle{definition}

\newcommand{\rk}{\operatorname{rk}}

\newcommand{\Skyt}{\operatorname{Skyt}}
\newcommand{\skyt}{\operatorname{skyt}}

\newcommand{\bSkyt}{\operatorname{\overline{Skyt}}}
\newcommand{\bskyt}{\operatorname{\overline{skyt}}}

\newcommand{\smdtemp}[3]{S_{#1,#2}(#3)}
\newcommand\smd{\smdtemp{m}{d}{\mcch}}

\newcommand{\coefftemp}[3]{c_{#1,#2}^{#3}(\mcch)}
\newcommand\coeff{\coefftemp{m}{d}{i}}

\newcommand{\wo}{\setminus} 
\newcommand{\compl}{^{c}}

\newcommand{\of}{\subseteq}

\newcommand{\adj}{\sim}

\newcommand{\jnd}{J(n,d)}
\DeclarePairedDelimiter\abs{|}{|}
\DeclarePairedDelimiter\set{\{}{\}}
\DeclarePairedDelimiter\parens{(}{)}

\begin{document}
\title[Kazhdan--Lusztig polynomials of Sparse Paving Matroids]{A combinatorial formula for Kazhdan--Lusztig polynomials of Sparse Paving Matroids}

\author[K. Lee]{Kyungyong Lee}
\thanks{K.L. was partially supported by the University of Alabama, University of Nebraska--Lincoln, Korea Institute for Advanced Study, and the NSF grant DMS 1800207.}
\address{Department of Mathematics, University of Alabama, Tuscaloosa, AL 35401 U.S.A., and School of Mathematics, Korea Institute for Advanced Study, Seoul 02455 Republic of Korea}
\email{kyungyong.lee@ua.edu; klee1@kias.re.kr}

\author[G. D. Nasr]{George D. Nasr}
\address{Department of Mathematics, University of Nebraska--Lincoln, Lincoln, NE 68588, U.S.A.}
\email{george.nasr@huskers.unl.edu}
\author[J. Radcliffe] {Jamie Radcliffe}
\thanks{Jamie Radcliffe was supported in part by Simons grant number 429383.}
\address{Department of Mathematics, University of Nebraska--Lincoln, Lincoln, NE 68588, U.S.A.}
\email{jamie.radcliffe@unl.edu}

\begin{abstract}
{ We prove the positivity} of Kazhdan-Lusztig polynomials for sparse paving matroids, which are known to be logarithmically almost all matroids, but are conjectured to be almost all matroids.  The positivity follows from a remarkably simple combinatorial formula we discovered for these polynomials using skew young tableaux. This supports the conjecture that Kazhdan-Lusztig polynomials for all matroids have {non-negative} coeffiecients. In special cases, such as uniform matroids, our formula has a nice combinatorial interpretation. 
\end{abstract}
\maketitle

\section{Introduction}
The Kazhdan--Lusztig polynomial of a matroid was introduced by Elias, Proudfoot, and Wakefield in 2016 \cite{firstklpolys}, which we define here. Throughout, let $M$ be a matroid, $F$ be a flat of the matroid $M$, $\rk$ be the rank function on $M$, and $\chi_M$ be the characteristic function for $M$. We denote $M^F$ (respectively $M_F$) for the localization (respectively contraction) for $M$ at $F$. Then, the Kazhdan-Lusztig polynomial for $M$, denoted $P_M(t)$ is given by the following conditions:
\begin{enumerate}
\item If $\rk M=0$, then $P_M(t)=1$.
\item If $\rk M>0$, then $\deg P_M(t)<{1\over 2} \rk M$. 
\item $\displaystyle t^{rk M}P_M(t\inv)=\sum_{F: \text{ a flat}} \chi_{M^F}(t)P_{M_F}(t)$.
\end{enumerate}

 Since their introduction, these polynomials have drawn active research efforts. Mostly, this is due to their (conjecturally) nice properties, such as non-negativity of coefficients, and real-rootedness (see \cite{firstklpolys,thag,klum,fanwheelwhirl,equithag}). There has also been much effort put into finding relations between these polynomials or generalizations thereof (see \cite{deletion,kls,pfia}). However, these polynomials have been explicitly calculated only for very special classes of matroids (for instance, see \cite{stirling,klum,klpolys,fanwheelwhirl,qniform}), and yet many of the known formulas have left much room for improvement. In particular, as of now, there is no enlightening interpretation for such coefficients. 
 
 In this paper, we provide a combinatorial formula for the Kazhdan-Lusztig polynomials of sparse paving matroids. Then we use this formula to deduce the positivity for the coefficients of these polynomials.
 
The class of sparse paving matroids is known to enjoy properties such as being dual-closed and minor-closed. However, what draws research interest to these matroids is a conjecture given by  Mayhew, Newman, Welsh, and Whittle \cite{sparsealmostall}. Based on Crapo's and Rota's prediction \cite{cr}, they conjecture that sparse paving matroids will eventually predominate in any asymptotic enumeration of matroids. That is,
 \[\lim_{n\to \infty}{s_n\over m_n}=1,\]
 where $s_n$ is the number of sparse paving matroids on $n$ elements and $m_n$ is the number of matroids on $n$ elements. {In pursuit of this conjecture, Pendavingh and van der Pol \cite{sparselog} have shown that }\[\lim_{n\to\infty} {\log s_n \over \log m_n}=1.\]
That is, so far, what we know is that logarithmically almost all matroids are sparse paving matroids.  
Hence, the fact that we are able to prove the non-negativity for coefficients of such matroids is favorable for the conjecture that all matroids have Kazhdan-Lusztig polynomials with non-negative coefficients. 
 
 There are several known characterizations of sparse paving matroids. Let $M$ be a matroid of rank $d$ so that the ground set has $m+d$ elements. Let $\mcb$ be the set of bases for $M$, so in particular $\mcb\subseteq {[m+d]\choose d}$. Set $\mcch:={[m+d]\choose d} \setminus \mcb$. Then $M$ is sparse paving if any (and hence all) of the following hold.
 \begin{enumerate}
 \item $\mcch$ is the set of circuit-hyperplanes for $M$.
 \item For distinct $C,C'\in \mcch$, we have $|C\sd C'|\geq 4$, where $C\sd C':=(C\setminus C') \cup (C'\setminus C)$ is the \textit{symmetric difference.}
 \item Every nonspanning circuit is a hyperplane.
 \item $M$ and its dual $M^*$ are both paving; that is, their circuits have cardinality at least $d$.
 \end{enumerate}
 
To this end, we let $\smd$ be the sparse paving matroid of rank $d$ with ground set $[m+d]$ so that $\mcch$ is the set of circuit-hyperplanes.
 
 The last thing we need to define before stating our main result is the object that will allow us to write our combinatorial formula for the coefficients of the polynomials. Define $\Skyt(a,i,b)$ to be the set of fillings of the following shape so that the rows and columns strictly increase with entries in $[a+b+2i-2]$.
  \begin{center}
  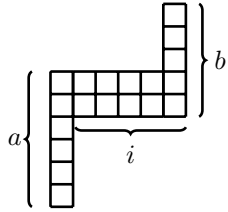
\begin{figure}[h]
  
\begin{tikzpicture}[scale=0.3, line width=1pt]
  \draw (-1,0) grid (0,6);
\draw[decoration={brace,raise=7pt},decorate]
 (-1,0) -- node[left=7pt] {$a$} (-1,6);
 
  \draw (0,4) grid (4,6);
\draw[decoration={brace,mirror, raise=4pt},decorate]
 (0.1,4) -- node[below=7pt] {$i$} (5,4);
 
  \draw (4,4) grid (5,9);
\draw[decoration={brace,mirror, raise=5pt},decorate]
 (5,4) -- node[right=7pt] {$b$} (5,9); 
 
\end{tikzpicture}
\caption{The left-most column has height $a$, followed by $i-1$ columns of height 2, followed by the right-most column of height $b$.}
  \end{figure}
  \end{center}
We define a related object which we denote $\bSkyt(i,b)$, the subset of $\Skyt(2,i,b)$ where the value 1 appears at the top of the left-most column. We set $\skyt(a,i,b):=\#\Skyt(a,i,b)$ and $\bskyt(i,b):=\#\bSkyt(i,b)$. There are some conventions for special values of $a$, $i$ and $b$, but we leave these for Section \ref{sec:SKYT}.

We are now ready to state our main result.

\begin{theorem}\label{thm:main}
Let $\coeff$ be the $i$-th coefficient for the Kazhdan-Lusztig polynomial for the sparse paving matroid $\smd$. Then
\[{\coeff}=\skyt(m+1,i,d-2i+1)-|\mcch|\cdot \bskyt(i,d-2i+1).\]
Moreover, this formula is always non-negative.
\end{theorem}

What is truly remarkable about this formula is that it is not effected by how the elements of $\mcch$ relate to one-another. Keep in mind that $\mcch$ could be \textit{any} set of elements so that their pairwise symmetric difference is at least 4. Given a fixed $m$, $d$, and $i$, the value of the coefficient is invariant of selection of $\mcch$ so long as $|\mcch|$ remains the same. 

When $\mcch$ is a disjoint collection, we have already shown in \cite[Proposition 2]{rhoremoved} that the formula in Theorem \ref{thm:main} has a manifestly positive interpretation.  
Consider the subset of $\Skyt(m+1,i,d-2i+1)$ satisfying at least	 one of the following three conditions.
\begin{itemize}
\item the top entry of the right-most column is 1; or

\item the bottom entry of the right-most column  is greater than $d+|\mcch|$; or

\item the third entry (from the top) of the left-most column  is less than $d+1$.
\end{itemize}
Then the size of this subset agrees with the formula we give in Theorem \ref{thm:main}. In the special case where $\mcch=\emptyset$, the second condition becomes tautological as the bottom of the right-most column is guaranteed to be at least $d+1$ for any tableaux. So when $\mcch=\emptyset$, we get the entire size of $\Skyt(m+1,i,d-2i+1)$ as our coefficient, as Theorem \ref{thm:main} indicates. Also in this case we have $\smd=U_{m,d}$, the uniform matroid of rank $d$ on $m+d$ elements. \footnote{The first (and only known) manifestly positive integral interpretation for uniform matroids was given in \cite[Remark 3.4]{klpolysequiv}, which requires possibly many Young diagrams.} 

In light of this, we have proven the following conjecture in the case of sparse paving matroids.

\begin{conjecture}
Let $M$ be a matroid of rank $d$ on $m+d$ elements, and let $c^i$ be the $i$-th coefficient for $P_M(t)$. Then \[c^i\leq c^i_{m,d}(\emptyset).\]
That is, among all matroids with rank $d$ and ground set size $m+d$, the Kazhdan-Lusztig polynomial for $U_{m,d}$ has the largest coefficients.
\end{conjecture}
\noindent  This conjecture was posed by Katie Gedeon. It has no written source, but was communicated to us to Nicholas Proudfoot.
 
 It is also interesting to note that when $\mcch$ is a disjoint collection, $\smd$ can be seen to be representable. This in turn gives a combinatorial formula for the intersection cohomology Poincar\'{e} polynomial of the corresponding reciprocal plane over a finite field, thanks to \cite{firstklpolys}. In general, though, almost all sparse paving matroids are not representable. This is due in large part to Nelson \cite{almostallnonrep} who showed that asymptotically almost all matroids are not representable. In particular, his work implies that the logarithmic growth of representable matroids is bounded by a polynomial. Meanwhile, the logarithmic growth of matroids in general are known to have at least exponential growth, and so the same must be true for sparse paving matroids. 

 One final thing to note that is interesting about our formula is that if $m+1=2$ or $d-2i+1=2$, then $\skyt(m+1,i,d-2i+1)$ becomes equal to a well-known number, namely the number of polygon dissections \cite{S}. Hence, when $m+1=d-2i+1=2$, it becomes a Catalan number.\footnote{This connection to polygon dissections was already mentioned in several places, namely in Remark 1.3 in \cite{intcohom} and Remark 5.3 of \cite{klpolys}, but with the discovery of our combinatorial object, this fact follows directly from \cite{S}.}

It should be remarked that in \cite{chowring}, Braden, Huh, Matherne, Proudfoot, and Wang say that their forthcoming paper will prove the non-negativity of the coefficients for the Kazhdan-Lusztig polynomials of all matroids. However, their approach may not develop a directly computable formula for these coefficients. On the other hand, we have formulas for the tableauxs appearing in Theorem \ref{thm:main}, which means one can use our formulas to directly find what these coefficients are in the case of sparse paving matroids.

This paper proceeds as follows. In section \ref{sec:SKYT}, we further discuss the elements of $\Skyt(a,i,b)$ and $\bSkyt(i,b)$. We also bring up some important conventions and useful identities for $\skyt(a,i,b)$ and $\bskyt(a,i,b)$. In section \ref{sec:flats_contr_local_charpoly}, we discuss flats, localizations, contractions, and characteristic polynomial for $\smd$. In section \ref{sec:KL for sparse paving}, we verify the formula for the Kazhdan-Lusztig polynomial of $\smd$ given in Theorem \ref{thm:main}. We then give some useful upper bounds on $|\mcch|$ in section \ref{sec:bounds}. We use these bounds to prove the non-negative part of Theorem \ref{thm:main}, which we do in section \ref{sec:positivity}. We end the paper with some integral identities we use throughout the paper in section \ref{sec:integral_identities}.

\textit{Acknowledgments:} The authors would like to thank Nicholas Proudfoot and Jacob Matherne for their helpful comments and feedback.
 
 \section{Skew Young Tableaux}\label{sec:SKYT}
 
 Consider the following shape.
 \begin{center}
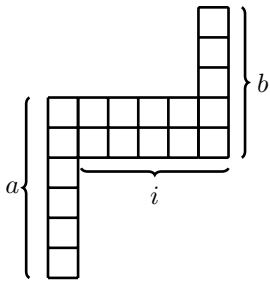
\begin{figure}[h]

\begin{tikzpicture}[scale=0.4, line width=1pt]
  \draw (-1,0) grid (0,6);
\draw[decoration={brace,raise=7pt},decorate]
 (-1,0) -- node[left=7pt] {$a$} (-1,6);
 
  \draw (0,4) grid (4,6);
\draw[decoration={brace,mirror, raise=4pt},decorate]
 (0.1,4) -- node[below=7pt] {$i$} (5,4);
 
  \draw (4,4) grid (5,9);
\draw[decoration={brace,mirror, raise=5pt},decorate]
 (5,4) -- node[right=7pt] {$b$} (5,9); 
\end{tikzpicture}
\caption{The left-most column has height $a$, followed by $i-1$ columns of height 2, followed by the right-most column of height $b$.}
\end{figure}
\end{center} 
A \textit{legal filling }of the above shape involves placing each number from $\{1,2,\dots, a+2i+b-2\}$ into the squares such that the values in the columns and rows strictly increase going down and right, respectively. Note that this is the same restriction on the entries of a standard young tableau, but the above shape does not fit the description of the typical young tableau. We refer to a legal filling of the above shape as a \textit{skew young tableau}, and denote $\Skyt(a,i,b)$ as the set of such legal fillings, and denote $\skyt(a,i,b):=\#\Skyt(a,i,b)$. 

For our tableaux to be defined, we need $a,b\geq 2$ and $i\geq 1$, but our formula in Theorem \ref{thm:main} may be used for other non-negative values of $a$, $b$, and $i$. Hence, there are some conventions we have set for the few exceptional values that can occur so that our formula still works. 
\begin{itemize}
\item If $i=0$, then $\skyt(a,i,b)=1$.
\item If $i>0$ and at least one of $a$ or $b$ is less than 2, then $\skyt(a,i,b)=0$.
\end{itemize} 

We also define a related collection of objects, which we denote $\bSkyt(i,b)$. This set is the subset of $\Skyt(2,i,b)$ so that 1 is always the entry at the top of the left-most column. The size of $\bSkyt(i,b)$ is denoted $\bskyt(i,b)$. By convention, $\bskyt(i,b)=0$ if $i=0$.

In \cite[Lemma 1]{rhoremoved}, we prove the following result.
\begin{lemma}\label{lem:skyt_count}
\[\skyt(a,i,b)={1\over i!(a-2)!(a+i-1)}\sum_{k=0}^{b-2} (-1)^k{a+b+2i-2 \choose b-2-k} {(a+2i+k)!(k+1)\over (a+i+k)(i+k+1)!}, \]
\end{lemma}

Using the proof of this result, it is not difficult to achieve the following identity by setting $a=2$ and replacing the $a+2i+b-2$ in Lemma \ref{lem:skyt_count} with $a+2i+b-3$. 

\begin{lemma}\label{lem:bsyt_count}
\[\bskyt(i,b)={1\over (i+1)!}\sum_{k=0}^{b-2} (-1)^k{b+2i-1 \choose b-2-k} {(2i+k+2)!(k+1)\over (i+k+2)!}, \]
\end{lemma}

One can achieve two formulas for $\skyt(a,i,b)$ and $\bskyt(i,b)$  that avoids alternating sums. We will need a few integral identities to produce these formulas. These identities can be found in section \ref{sec:integral_identities}, but are referenced as they are needed in the proofs that follow. Throughout, $(x)^{(n)}$ is the rising factorial $(x)(x+1)\cdots(x+n-1)$ for integers $x$ and $n$.

We start with the formula for $\skyt(a,i,b)$.

\begin{lemma}\label{lem:manifposskyt}
\[\skyt(a,i,b)={a+i-2\choose i}{a+b+2i-2\choose b+i-1}\sum_{k=0}^{b-2}{{b+i-k-3\choose i-1}\over {a+i+k\choose k+1}}\]
\end{lemma}

\begin{proof}
One can rewrite Lemma \ref{lem:skyt_count} as
\begin{align}
\skyt(a,i,b)&={(a+b+2i-2)!\over i!(a-2)!(a+i-1)(b-2)!}\sum_{k=0}^{b-2} (-1)^k {b-2\choose k}{1\over (a+i+k)(k+2)^{(i)}}.\label{eq:skytalt}
\end{align}

We can recover this sum for $\skyt(a,i,b)$ by applications of integrals to a polynomial. Let \[\displaystyle f(x,y)={(a+b+2i-2)!xy^{a+i-1}(1-xy)^{b-2}\over i!(a-2)!(a+i-1)(b-2)!}.\] 
Our integrals our broken up into three parts.
\begin{enumerate}
\item[(a)] First find $g(x)$, where $g(x):=\displaystyle\int_0^1 \ f(x,y) \ dy$; then
\item[(b)] find $\displaystyle h_{i-1}(x_{i-1}):=\int_0^{x_{i-1}}h_{i-2}(x_{i-2})\ dx_{i-2}$, where $\displaystyle h_1(x_1):=\int_0^{x_1} g(x_0) \ dx_0$ and $x_0,x_1,\dots, x_{i-1}$ are $i$ variables; then 
\item[(c)] solve $\displaystyle\int_0^1 h_{i-1}(x_{i-1})\ dx_{i-1}$.
\end{enumerate} 

It is not difficult to show that, if $(1-xy)^{b-2}$ is written using the binomial expansion, part (c) will give the equation for $\skyt(a,i,b)$ found in equation \eqref{eq:skytalt} above. To get the statement of Lemma \ref{lem:manifposskyt}, we apply these three steps to $f(x,y)$ directly as written.

First, we use Corollary \ref{cor:int_0^1_y^a(1-xy)^b} to do part (a).
\[g(x):=\int_0^1f(x,y)\ dy={(a+b+2i-2)!(a+i-1)!\over i!(a-2)!(a+i-1)}\sum_{k=0}^{b-2}{(1-x)^{b-k-2}x^{k+1}\over (a+i+k)!(b-k-2)!}.\]

To complete parts (b) and (c) we apply Proposition \ref{prop:int_i_times} to get

\[\int_0^1 h_{i-1}(x_{i-1}) \ dx_{i-1} = {(a+b+2i-2)!(a+i-1)!\over i!(a-2)!(a+i-1)(i-1)!(b+i-1)!}\sum_{k=0}^{b-2}{(b+i-k-3)!(k+1)!\over (a+i+k)!(b-k-2)!}\]

This gets us a manifestly positive sum, and all that is left to get our desired result is to perform some algebraic manipulations. One can combine the terms $(b+i-k-3)!$, $(b-k-2)!$, and $(i-1)!$ combine to give $\displaystyle {b+i-k-3 \choose i-1}$. Then combine $(a+i-1)!$, $(k+1)!$, and $(a+i+k)!$ to get $\displaystyle {a+i+k\choose k+1}$. Then scale by ${(a+i-2)!\over (a+i-2)!} $ allows us to group the remaining factors into binomial coefficients giving 
\[{a+i-2\choose i}{a+b+2i-2\choose b+i-1}\sum_{k=0}^{b-2}{ {b+i-k-3 \choose i-1}\over {a+i+k\choose k+1}}. \qedhere\]
\end{proof}

\begin{rem} \label{rem:rewriting}
\leavevmode
\begin{enumerate}
\item While having a manifestly positive formula for $\skyt(a,i,b)$ is nice, it is unfortunate that, in general, the terms of the sum in Lemma \ref{lem:manifposskyt} are not necessarily integers, even if you scale them by ${a+i-2\choose i}$ and ${a+b+2i-2 \choose b+i-1}$. 
\item It will be useful to rewrite Lemma \ref{lem:manifposskyt} using a common denominator. We can do this by rewriting the binomials in the sum using the falling factorial $(x)_{(n)}:=x(x-1)\cdots(x-n+1)$. Rewriting the sum gives

\begin{align*}
&\sum_{k=0}^{b-2}{{b+i-k-3\choose i-1}\over {a+i+k\choose k+1}}\\
&=\sum_{k=0}^{b-2}{(b+i-k-3)_{b-k-2}(k+1)!\over (b-k-2)! (a+i+k)_{(k+1)}}\\
&={1\over (b-2)! (a+i+b-2)_{(b-1)}}\sum_{k=0}^{b-2}(b+i-k-3)_{b-k-2}(k+1)!(b-2)_{(k)}(a+i+b-2)_{(b-k-2)}
\end{align*}
We will find this version useful later, though it is not as concise as the original formula.
\end{enumerate}
\end{rem}

Using similar methods, we can find a formula for $\bskyt(i,b)$ which not only avoids an alternating sum, but is in fact a single term.
\begin{lemma}\label{lem:manifposbarskyt}
\[\bskyt(i,b)={2(b+2i-1)!\over (i+1)!(i-1)!(b-2)!(b+i)(b+i-2)}\]
\end{lemma}
\begin{proof}
One can rewrite Lemma \ref{lem:manifposbarskyt} as 
\begin{align}\label{eq:alt_bskyt}
\bskyt(i,b)={(b+2i-1)!\over (i+1)!(b-2)!}\sum_{k=0}^{b-2}(-1)^k{b-2\choose k} {(2i+k+2)\over (k+2)^{(i+1)}}.
\end{align}
We can recover this sum for $\bskyt(a,i,b)$ by applications of a derivative and integrals to a polynomial. Let
\[\displaystyle f(x,y)={(b+2i-1)!xy^{2i+2}(1-xy)^{b-2}\over (i+1)!(b-2)!}.\] 

We break up our plan for applications of a derivative and integrals into three parts. 

\begin{enumerate}
\item[(a)] First solve $g(x):=\displaystyle\left. {d\over dy} f(x,y) \ \right|_{y=1}$; then
\item[(b)] find $\displaystyle h_{i}(x_{i}):=\int_0^{x_{i}}h_{i-1}(x_{i-1})\ dx_{i-1}$, where $\displaystyle h_1(x_1):=\int_0^1 g(x_0) \ dx_0$ and $x_0,x_1,\dots, x_i$ are $i+1$ variables; then finally
\item[(c)] find $\displaystyle\int_0^1 \ h_i(x_i) \ dx_i$.
\end{enumerate} 
If one writes $(1-xy)^{b-2}$ using the binomial expansion, part (c) outputs the equation for $\bskyt$ found in equation \eqref{eq:alt_bskyt} above. We claim that leaving $f(x,y)$ as written and then applying these three steps lead to the statement of Lemma \ref{lem:manifposbarskyt}.

 First, for part (a) observe that 
\[g(x)=\left.{d\over dy} f(x,y)\right|_{y=1}={2(i+1)(b+2i-1)!\over (i+1)!(b-2)!}x(1-x)^{b-2}-{(b-2)(b+2i-1)!\over (i+1)!(b-2)!}x^2(1-x)^{b-3}.\]

We do parts (b) and (c) simultaneously due to Proposition \ref{prop:int_i_times}. This gives 
\begin{align*}
\int_0^1 h_i(x_i)\ dx_i&={2(i+1)(b+2i-1)!\over (i+1)!(b-2)!}{(b-2+i)!\over i!(b+i)!}-{(b-2)(b+2i-1)!\over (i+1)!(b-2)!}{2(b-3+i)!\over i!(b+i)!}\\
&={2(b+2i-1)!(b+i-3)![(i+1)(b-2+i)-(b-2)]\over  i!(i+1)!(b+i)!(b-2)!}\\
&={2i(b+2i-1)!(b+i-3)!(b+i-1)\over i!(i+1)!(b+i)!(b-2)!}\\
&={2(b+2i-1)!\over (i+1)!(i-1)!(b-2)!(b+i)(b+i-2)} \qedhere
\end{align*}

\end{proof}

\section{Flats, Contractions, Localizations, and Characteristic Polynomials for $\smd$}\label{sec:flats_contr_local_charpoly}\leavevmode

Throughout, let $F$ be a flat, that is, a set which is maximal with respect to its rank. For a matroid $M$, recall that $M^F$ (respectively, $M_F$) denotes the localization (respectively, contraction) of $M$ at $F$. By $M^F$, we mean the matroid with ground set $F$, whose independent sets are those subsets of $F$ that are also independent in $M$. By $M_F$, we mean the matroid with ground set $M\setminus F$, whose independent sets are those subsets whose union with a basis for $F$ is independent in $M$. 

First, we discuss the flats of $\smd$. It is an elementary exercise to verify the following.

\begin{prop}
The flats of $\smd$ are 
\begin{enumerate}
\item the sets of cardinality at most $d-2$;
\item the sets of cardinality $d-1$ not contained in any element of $\mcch$;
\item the elements of $\mcch$;
\item $[m+d]$.
\end{enumerate}
\end{prop}

With this, we can now discuss the localizations and contractions of $\smd$. First, recall the localizations and contractions of $U_{m,d}$, the uniform matroid of rank $d$ with groundset $[m+d]$.
\[(U_{m,d})^F=\begin{cases} U_{m,d} & F=[m+d]\\ U_{0,|F|} & F\neq [m+d] \end{cases},\]
and 
\[(U_{m,d})_F=\begin{cases} U_{0,0} & F=[m+d]\\ U_{m,d-|F|} & F\neq [m+d] \end{cases}.\]


The corresponding equations for $\smd$ can also be described in a similar manner. In what follows, if $F$ is a flat, then we define $\mcch(F):=\{C\setminus F: C\in \mcch \text{ such that }F\subseteq C\}$. It is worth noting that if $\mcch$ is the set of circuit-hyperplanes for a sparse paving matroid, then so is $\mcch(F)$, so long as $F$ is strictly contained in some circuit-hyperplane. One way to check this is by verifying $\mcch(F)$ satisfies the condition that any pair has symmetric difference at least 4.

\begin{prop}\label{prop:rest_local_orum}
\[\smd^F=\begin{cases} \smd & F=[m+d]\\ U_{1,d-1} & F\in \mcch \\ U_{0,|F|} & \textit{otherwise}  \end{cases}\]
and 
\[\smd_F=\begin{cases} \smd & F=\emptyset\\  U_{m-1,1} & F\in \mcch\\ \smdtemp{m}{d-|F|}{\mcch(F)} & \emptyset\subsetneq F\subsetneq C, \text{ for some $C\in \mcch$}\\ (U_{m,d})_F & \textit{otherwise.}  \end{cases}\]

\end{prop}

\begin{proof}
For the localization, the only new case necessary to mention in comparison to the uniform case is for $F\in \mcch$; the other cases follow from the uniform case. The localization of this matroid at $F$ treats $F$ as the ground set, with independent sets being those that are independent in $\smd$. We know every \textit{proper} subset of $F$ is independent, giving $ U_{1,d-1}$.

Now for the contraction. If we have $F\nsubseteq C$ for all $C\in \mcch$, then the structure of $\smd_F$ is exactly that of $(U_{m,d})_F$. 
For the case where $F\in \mcch$, we want the subsets of $S:=[m+d]\setminus F$ such that their union with a basis for $F$ is independent in $\smd$. The bases for $F$ are the elements of ${F\choose d-1}$. Note if $B\in {[m+d]\choose d}$ satisfies $|B\triangle F|=2$, then $B$ is independent in $\smd$. This means the desired subsets of $S$ are the empty set and every singleton of $S$. This gives a matroid isomorphic to $U_{m-1,1}$. 
Finally, when $\emptyset \subsetneq F\subsetneq C$, for some $C\in \mcch$, note that $F$ is independent, and hence a basis for itself. Thus, the independent sets for $\smd_F$ are the subsets $X$ of $[m+d]\setminus F$ so that $X\cup F$ is independent in $\smd$. That is, $|X|\leq d-|F|$. When $|X|<d-|F|$, $|X\cup F|<d$ and every subset of $[m+d]$ of size smaller than $d$ is independent. When $|X|=d-|F|$, $X\cup F$ is a basis for $\smd$ if and only if $X\cup F\neq C$, for any $C\in \mcch$, which is true if and only if $X\notin \mcch(F)$. That is, we get a matroid isomorphic to $\smdtemp{m}{d-|f|}{\mcch(F)}$. 
\end{proof}

With these in mind, we can now compute the characteristic equation for all localizations for $\smd$. However, by Proposition \ref{prop:rest_local_orum}, we equivalently just need to find the characteristic polynomial for $U_{m,d}$ and $\smd$. 

First, recall that for a matroid $M$, the characteristic polynomial is given by 
\[\chi_M(t)=\sum_{F\in L(M)} \mu_{L(M)}(\hat{\textbf{0}},F)t^{\rk M -\rk F},\]
where $L(M)$ is the lattice of flats for matroid $M$. The case when $M=U_{m,d}$, $\chi_M(t)$ is well understood. 
\[\chi_{U_{m,d}}(t)=(-1)^d{m+d-1\choose d-1}+\sum_{i=0}^{d-1} (-1)^i{m+d\choose i}t^{d -i}.\]

Parts of this also arise in $\chi_{\smd}$. 
\begin{prop}\label{prop:charpoly} Let $c=|\mcch|$.
\[\chi_{\smd}(t)=(-1)^d{m+d-1\choose d-1}-(-1)^dc+t(-1)^{d-1}\left({m+d\choose d-1}-c\right)+\sum_{i=0}^{d-2} (-1)^i{m+d\choose i}t^{d -i}.\]
\end{prop}
\noindent It is noteworthy that this characteristic polynomial is the same for all choices of $\mcch$ that have the same size. This is due entirely to the symmetric difference condition on $\mcch$, as we will utilize in the proof.

\begin{proof}[Proof of Proposition \ref{prop:charpoly}.]
For convenience, we omit subscripts for $\chi$ and $\mu$, since throughout we work in $\smd$.
The terms of degree at least 2 follows from the uniform matroid case since in $\smd$, every set of size at most $d-2$ is still flat, since every set of size $d-1$ is independent. The term of degree one comes from summing $\mu(\zb, F)$ for flats $F$ of rank $d-1$. Recall that these flats are the elements of $\mcch$ and all elements of ${[m+d]\choose d-1}$ not contained in any member of $\mcch$. When $F$ is one of the latter described flats, it follows from the uniform case that $\mu(\zb, F)=(-1)^{d-1}$. Note that the number of such flats is ${m+d\choose d-1}-c{d\choose d-1}$, since the symmetric difference condition on $\mcch$ implies that $|C_i\cap C_j|\leq d-2$ for all $C_i,C_j\in \mcch$. That is to say that no set of size $d-1$ is contained in two elements of $\mcch$. Otherwise, if $C\in \mcch$,
\begin{align*}
\mu(\zb,C)&=-\sum_{\zb\leq F<C}\mu(\zb,F)\\
&=-\sum_{i=0}^{d-2}(-1)^i{m+d\choose i}\\
&=(-1)^d+d(-1)^{d-1}.
\end{align*}
Thus the coefficient linear term for $\chi$ is given by 
\begin{align*}
c(-1)^d+c d(-1)^{d-1}+(-1)^{d-1}\left({m+d\choose d-1}-c{d\choose d-1}\right)=(-1)^{d-1}{m+d\choose d-1}-c(-1)^{d-1}.
\end{align*}

For the constant term, it is equivalent to negate the sum over $\mu(\zb, F)$ for all flats $F\neq [m+d]$. This gives 
\begin{align*}
-\sum_{i=0}^{d-2} (-1)^i{m+d\choose i}-(-1)^{d-1}{m+d\choose d-1}-c(-1)^d&=-\sum_{i=0}^{d-1} (-1)^i{m+d\choose i}-c(-1)^d\\
&=(-1)^d{m+d-1\choose d-1}-c(-1)^d. \qedhere
\end{align*}
\end{proof}

It will be helpful to restate this proposition in the following way for when we prove Theorem \ref{thm:main}.
\begin{prop}\label{prop:charpoly_restated} (Proposition \ref{prop:charpoly} restated.)
\[ [t^i]\chi_{\smd} = \begin{cases} (-1)^d{m+d-1\choose d-1}-c(-1)^d & i=0\\
                                           (-1)^{d-1}{m+d\choose d-1}-c(-1)^{d-1} & i=1\\
                                           (-1)^{d-i}{m+d\choose d-i} & 2\leq i\leq d \end{cases}\]
\end{prop}

\section{The Kazhdan-Lusztig Polynomials for Sparse Paving Matroids}\label{sec:KL for sparse paving}
This section is dedicated to justifying the combinatorial formula given in Theorem \ref{thm:main}. We restate this part here for convenience, as its own Theorem. 

\begin{theorem}\label{thm:combformula}
Let $\coeff$ be the $i$-th coefficient for the Kazhdan-Lusztig polynomial for the sparse paving matroid $\smd$. Then
\[{\coeff}=\skyt(m+1,i,d-2i+1)-|\mcch|\cdot \bskyt(i,d-2i+1).\]
\end{theorem}

\begin{rem}\label{rem:conventions}
For some values of $m$, $d$, and $i$, we need to use our conventions set in place for $\skyt(a,i,b)$ and $\bskyt(a,i,b)$ in section \ref{sec:SKYT} for our formula to truly work.
\begin{itemize}
\item \cite[Proposition 2.11]{firstklpolys} shows that the degree 0 term always has coefficient 1. That is, when $i=0$, our formula must always return 1. 
\item When $d=0$ we are forced to have $P_{\smd}(t)=1$. 
\item When $0<d<3$, the degree requirement on Kazhdan-Lusztig polynomials forces $P_{\smd}(t)$ to have degree 0. Namely, in this case, we have $P_{\smd}(t)=1$, again by \cite[Proposition 2.11]{firstklpolys}. 
\item When $m=0$, note that $\mcch$ is forced to be empty and $\smd$ becomes $U_{0,d}$. It is shown in  \cite[Proposition 2.7]{firstklpolys} that $P_{M_1\oplus M_2}(t)=P_{M_1}(t)P_{M_2}(t)$ for matroids $M_1$ and $M_2$. With this, one can verify that $P_{U_{0,d}}(t)=1$ by seeing that $P_{U_{0,1}}(t)=1$ based on the $d<3$ discussion above. 
\end{itemize}
In all cases, our conventions guarantee we get the right values. Besides these cases, our conventions are not needed for our formula, and we are guaranteed that $\smd$ has more interesting structure than that of the boolean lattice.

\end{rem} 

The following technical result will be crucial in demonstrating why the formula given in Theorem \ref{thm:combformula} only depends on $|\mcch|$, and not the relationship between the elements of $\mcch$. 

\begin{lemma}\label{lem:simplification} Let $c,i\in \N\cup\{0\}$. For $I\subseteq [c]$, let $x_I$ be a variable. Let $g(k)$ and $h(k)$ are functions varying in $k$. Then 
\[-\sum_{\substack{J\subseteq [c] \\ |J|\geq 2 }} (-1)^{|J|} x_J\sum_{k=0}^i g(k)=\sum_{\substack{\emptyset\subsetneq I\subseteq [c] }}\sum_{\substack{I\subseteq J\subseteq [c] \\ |J|\geq 2 }} (-1)^{|J|-|I|} x_J\sum_{k=0}^i (\ g(k)-|I|h(k)\ ),\]

\end{lemma}

\begin{proof}
We show that the term with $x_J$ on both sides of the statement of the lemma is the same for every $J\subseteq [c]$, where $|J|\geq 2$. We start with the coefficient of $x_J$ on the right side. We note that the terms with $x_J$ appear for each $I$ that is contained in $J$, where $|I|\geq 1$. Hence, the term with $x_J$ on the right hand side of the statement of the Lemma is
\begin{align*}
&\sum_{\ell=1}^{|J|}  (-1)^{|J|-\ell}x_J{|J| \choose \ell} \sum_{k=0}^i (\ g(k)-\ell h(k)\ )\\
&=x_J(-1)^{|J|}\sum_{\ell=1}^{|J|} {|J| \choose \ell}(-1)^{\ell}\sum_{k=0}^i (\ g(k)-\ell h(k)\ )\\
&=x_J(-1)^{|J|}\left(\sum_{k=0}^i g(k)\sum_{\ell=1}^{|J|} (-1)^{\ell}{|J| \choose \ell}-\sum_{k=0}^i h(k) \sum_{\ell=1}^{|J|} (-1)^{\ell}\ell{|J|\choose \ell} \right)\\
&=x_J(-1)^{|J|}\left(-\sum_{k=0}^i g(k)\right),\\
\end{align*}
since we know in general we have the identities $\displaystyle \sum_{\ell=0}^n (-1)^\ell{n\choose \ell}=0$ for $n\geq 1$ and $\displaystyle\sum_{\ell=0}^n(-1)^\ell {n\choose \ell}\ell=0$ for $n\geq 2$. Note that the there is exactly one time where $x_J$ appears exactly once, and the corresponding term is $\displaystyle-x_J(-1)^{|J|}\sum_{k=0}^i g(k)$. 
\end{proof}

We now prove the desired formula for $\coeff$.
\begin{proof}[Proof of Theorem \ref{thm:combformula}.]
Let $M:=\smd$, and set $c:=|\mcch|$. Recall that the definition for the Kazhdan-Lusztig polynomial is that it satisfies the following recurrence,
\[\displaystyle t^{rk M}P_M(t\inv)=\sum_{F \text{ a flat}} \chi_{M^F}(t)P_{M_F}(t),\]
which may be rewritten as 
\[\displaystyle t^{rk M}P_M(t\inv)-P_M(t)=\sum_{F \text{ a non-empty flat}} \chi_{M^F}(t)P_{M_F}(t).\]

Recall that $\deg P(t)<{1\over 2} d$, and so the power of each monomial in $t^{d}P_M(t\inv)$ is strictly larger than ${1\over 2}d$. Hence, our goal is to show that for $0\leq i< {1\over 2} d$ we have 
\begin{align}\label{eq:step1}
-\skyt(m+1,i,d-2i+1)+c\cdot \bskyt(i,d-2i+1)=[t^i]\sum_{F \text{ a non-empty flat}} \chi_{M^F}(t)P_{M_F}(t).
\end{align}

Using our work from Proposition \ref{prop:rest_local_orum}, and consolidating common factors involving the various flats in $\mcch$, we can rewrite the right of equation \eqref{eq:step1} to be 
\begin{align}\label{eq:step2}
[t^i]\chi_{\smd}+c[t^i]\chi_{U_{1,d-1}}P_{U_{m-1,1}}+\sum_{\substack{\emptyset\subsetneq F\subsetneq C\\\text{For some $C\in \mcch$}}}[t^i]\chi_{U_{0,|F|}}P_{\smdtemp{m}{d-|F|}{\mcch(F)}}+\sum_{\substack{\emptyset\subsetneq F\subsetneq [m+d]\\ F\nsubseteq C\ \forall C\in \mcch}}[t^i]\chi_{U_{0,|F|}}P_{{U_{m,d-|F|}}},
\end{align}
where the first term corresponds to the case where $F=[m+d]$, and the second where $F\in \mcch$.

By Proposition \ref{prop:charpoly_restated}, we are required to break this up into three case: $i=0$, $i=1$, and $2\leq i<d/2$ if we are to write this out explicitly. Note we can write everything explicitly except $P_{{\smdtemp{m}{d-|F|}{\mcch(F)}}}$. Hence, we proceed by induction on the matroid rank $d$, noting that $d>d-|F|$ since for the corresponding summand $F$ is never empty.

We now define some notation in order to rewrite the summations appearing in \eqref{eq:step2}. Let $I\subseteq [c]$ and $C_i\in \mcch$. We define $\displaystyle C_I:=\bigcap_{i\in I}C_i$ and denote $c_I:=|C_I|$. By convention, $C_\emptyset=[m+d]$. Recall that $\mcch(F):=\{C\setminus F: C\in \mcch \text{ such that } F\subseteq C\}$. Let $j$ be an integer and define the following sum indexed by $J$:
\[\Phi_j(I):=\sum_{I\subseteq J \subseteq [c]}(-1)^{|J|-|I|}{c_J\choose j}.\]
If $j$ is selected appropriately, $\Phi_j(I)$ counts the number of flats of rank $j$ contained in $C_I$, but not in any $C_J$ so that $C_J\subseteq C_I$. Hence, $F$ is a flat counted by $\Phi_j(I)$ if and only if $\mcch(F)=\{C_i\setminus F: i\in I\}$. What we will leverage from this is that $|\mcch(F)|=|I|$.


We can now rewrite equation \eqref{eq:step2}. We use the Kronecker delta function $\delta(i,j)=\begin{cases} 1& i=j\\ 0& i\neq j\\\end{cases}$ to combine the cases for $i=1$ and $2\leq i<d/2$. 
\begin{enumerate}
\item[$i=0$:] 
\begin{align*}
&(-1)^d{m+d-1\choose d-1}-c(-1)^{d}+c(-1)^{d-1}{d-1\choose d-2}\\
&+\sum_{j=1}^{d-2}\sum_{\emptyset\subsetneq I \subseteq [c]}\Phi_j(I)(-1)^j(\skyt(m+1,0,d-j+1)-|I|\cdot\bskyt(0,d-j+1))\\
&+\sum_{j=1}^{d-1}\Phi_j(\emptyset)(-1)^j\skyt(m+1,0,d-j+1)
\end{align*}

\item[$i>0$:]\begin{align*}
&(-1)^{d-i}{m+d-1\choose d-i}-c(-1)^{d-1}\delta(i,1)+c(-1)^{d-1-i}{d\choose d-1-i}\\
&+\sum_{j=1}^{d-2}\sum_{\emptyset\subsetneq I \subseteq [c]}\Phi_j(I)\sum_{k=0}^i(-1)^{j-i+k}{j\choose j-i+k}(\skyt(m+1,k,d-j-2k+1)-|I|\bskyt(k,d-j-2k+1))\\
&+\sum_{j=1}^{d-1}\Phi_j(\emptyset)\sum_{k=0}^i(-1)^{j-i+k}{j\choose j-i+k}\skyt(m+1,k,d-j-2k+1)
\end{align*}
\end{enumerate}

In both cases, the sum running from $j=1$ to $j=d-2$ is the summand in equation \eqref{eq:step2} over $\emptyset\subsetneq F\subsetneq C$ for $C\in \mcch$, since the flats contained in $C$ have size at most $d-2$.  The other sum running from $j=1$ to $j=d-1$ corresponds to the summand in equation \eqref{eq:step2} over $\emptyset\subsetneq F\subsetneq [m+d]$ such that $F\nsubseteq C$ for all $C\in \mcch$. 

To simplify things further, first, note that \[\displaystyle\Phi_{d-1}(\emptyset)={m+d\choose d-1}-c{d\choose d-1}.\] By construction, $\Phi_{d-1}(\emptyset)$ counts the rank $d-1$ flats contained in no element of $\mcch$. Recall that the only rank $d-1$ flats are those not contained in any circuit-hyperplane.

Next, note that many terms from the two sums running over $j$ in both the $i=0$ and $i>0$ case will cancel as a result of Lemma \ref{lem:simplification}. Fix $j\leq d-2$ and suppose $J\subseteq [c]$. Set

$\bullet$ $\displaystyle x_J:={c_J\choose j}$,

$\bullet$ $\displaystyle g(k):=(-1)^{j-i+k}{j \choose j-i+k}\skyt(m+1,k,d-j-2k+1)$, and

$\bullet$ $\displaystyle h(k):=(-1)^{j-i+k}{j\choose j-i+k}\bskyt(k,d-j-2k+1)$.

\noindent This allows us to rewrite our two cases in the following way. 
\begin{enumerate}
\item[$i=0$:] 
\begin{align*}
&(-1)^d{m+d-1\choose d-1}-c(-1)^{d}+c(-1)^{d-1}{d-1\choose d-2}\\
&+\sum_{j=1}^{d-2}\sum_{\emptyset\subsetneq I \subseteq [c]}
\sum_{I\subseteq J \subseteq [c]}(-1)^{|J|-|I|}x_J(\ g(0)-|I|h(0)\ )\\
&+\sum_{j=1}^{d-1}\sum_{\emptyset\subseteq J \subseteq [c]}(-1)^{|J|}x_J g(0)
\end{align*}

\item[$i>0$:]\begin{align*}
&(-1)^{d-i}{m+d\choose d-i}-c(-1)^{d-1}\delta(i,1)+c(-1)^{d-1-i}{d\choose d-1-i}\\
&+\sum_{j=1}^{d-2}\sum_{\emptyset\subsetneq I \subseteq [c]}\sum_{I\subseteq J \subseteq [c]}(-1)^{|J|-|I|}x_J\sum_{k=0}^ig(k)-|I|h(k)\\
&+\sum_{j=1}^{d-1}\sum_{\emptyset\subseteq J \subseteq [c]}(-1)^{|J|}x_J\sum_{k=0}^ig(k)
\end{align*}
\end{enumerate}

The following argument works for both the $i=0$ and $i>0$ case, so we speak of both simultaneously as if they were one. Let $A$ correspond to the sum indexed by $j$ where $j$ is at most $d-2$. Likewise define $B$ to be the sum indexed by $j$ where $j$ is at most $d-1$. By Lemma \ref{lem:simplification}, the terms where $|J|\geq 2$ in $A$ will cancel all terms where $|J|\geq 2$ in $B$. What remains in $A$ are the terms where $|J|=1$, that is, the terms where $J=I$ and $|I|=1$. There are $c$ such terms, each contributing ${d\choose j}$, as the members of $\mcch$ have cardinality $d$. For $B$, when $j\leq d-2$, the only terms that remain are those where $|J|$ equals 0 or 1. This gives $c+1$ terms: one contributing ${m+d\choose j}$, and $c$ terms contributing $-{d\choose j}$. Combining this with our identity for $\Phi_{d-1}(\emptyset)$ given above, we get the following simplification.

\begin{enumerate}
\item[$i=0$:] 
\begin{align*}
&(-1)^d{m+d-1\choose d-1}-c(-1)^{d}+c(-1)^{d-1}{d-1\choose d-2}\\
&+c\sum_{j=1}^{d-2}{d\choose j}(-1)^j(\skyt(m+1,0,d-j+1)-\bskyt(0,d-j+1))\\
&+\sum_{j=1}^{d-1}\left({m+d\choose j}-c{d\choose j}\right)(-1)^j\skyt(m+1,0,d-j+1)
\end{align*}

\item[$i>0$:]\begin{align*}
&(-1)^{d-i}{m+d\choose d-i}-c(-1)^{d-1}\delta(i,1)+c(-1)^{d-1-i}{d\choose d-1-i}\\
&+c\sum_{j=1}^{d-2}{d\choose j}\sum_{k=0}^i(-1)^{j-i+k}{j\choose j-i+k}(\skyt(m+1,k,d-j-2k+1)-\bskyt(k,d-j-2k+1))\\
&+\sum_{j=1}^{d-1}\left({m+d\choose j}-c{d\choose j}\right)\sum_{k=0}^i(-1)^{j-i+k}{j\choose j-i+k}\skyt(m+1,k,d-j-2k+1).
\end{align*}
\end{enumerate}

We now point out that remarkably, this formula no longer depends on the structure of $\mcch$, only the size. Hence, the proof proceeds as in the case of Theorem 3 in \cite{rhoremoved}.
\end{proof}

\section{Bounds on $|\mcch|$} \label{sec:bounds}
Our proof for the non-negativity of Theorem \ref{thm:main} will be purely computational. Hence, since $|\mcch|$ is a part of our formula, having bounds on this value will be useful. We will give two particularly important bounds. 

The first bound is given as follows.
\begin{theorem}\label{thm:codingbound}
\[|\mcch|\leq {1\over m+1}{m+d\choose d}.\]
\end{theorem}

This can be recovered in multiple settings. One can find an outline of a matroid theory argument in \cite[Lemma 2.7]{sparsebasispaper}. However, this bound also happens to be a standard coding theory result. Recall that for $\smd$, the circuit-hyperplanes $\mcch$ is a subset of elements in ${[m+d]\choose d}$ so that any pair has symmetric difference at least 4. One could equivalently describe such a set as a binary constant-weight code with hamming distance 4. In this context, the bound in Theorem \ref{thm:codingbound} gives a bound on the size of a code with these conditions, as shown in \cite[Theorem 12]{codingbounds}. In fact, \cite{codingbounds} proves a more arbitrary bound accounting for any lower bound on symmetric difference, not just 4. It is also worth noting that the proofs for this bound given in both \cite{codingbounds} and \cite{sparsebasispaper} are in fact different, even when both are in the language of matroid theory.

While this bound will serve useful, there will be times where it will not be sufficient for our purposes. Unlike the prior bound, we found no literature to support the bound that follows.

\begin{theorem}\label{thm:jamiesbound}
\[|\mcch|\leq {2\over m+d+2}{m+d\choose d}.\]
\end{theorem}

\begin{rem}
These two bounds have an interesting relationship. First, observe that \[ {1\over m+1}{m+d\choose d}> {2\over m+d+2}{m+d\choose d} \text{ if and only if } d>m.\]
A take-away here is that both bounds are necessary to get a good bound for $|\mcch|$. Excitingly, when $m=d$, not only do these bounds agree, but they equal the $m$th Catalan number $C_m$, where
\[C_m={1\over m+1}{2m\choose m}.\]

\end{rem}

To prove Theorem \ref{thm:jamiesbound}, we will utilize a graph theory technique known as \textit{discharging}. First, though, it is necessary to make clear the connection between sparse paving matroids and graphs. Let $J(n,d)$ be a graph with vertex set ${[n]\choose d}$, where vertices are adjacent if and  only if their symmetric difference is size 2. This graph is best known as the \textit{Johnson Graph}. The symmetric difference condition on $\mcch$ implies that $\mcch$ is an independent set in $J(m+d,d)$, that is, a set of vertices with no edges between them. So finding an upper bound on $|\mcch|$ is equivalent to a bound on the size of an independent set in $J(m+d,d)$.

There are some final graph theory notation conventions we give before providing the proof of Theorem \ref{thm:jamiesbound}. Let $A$ and $B$ be vertices in $J(n,d)$. To indicate $A$ and $B$ are {adjacent} we write $A\adj B$. When an edge has vertex $A$ as an endpoint, we say that edge is \textit{incident} to $A$. By $N(A)$ we mean the induced graph on the vertices adjacent to $A$ in $J(n,d)$. That is, $N(A)$ is the subgraph of $\jnd$ where for all vertices $B,C\in N(A)$, we have $B\adj C$ in $N(A)$ if and only if $B \adj C$ in $J(n,d)$.

\textit{Proof of Theorem \ref{thm:jamiesbound}.}
{ 
  Let $I \of \binom{[n]}d$ be an independent set of vertices in $\jnd$. We will describe an assignment of weights to edges of $\jnd$ based on $I$. Start with a weight of 0 on all edges of $\jnd$. If $A\in I$ we add a weight of $1$ to each edge incident with $A$. Furthermore, $A$ adds a weight of $1/2$ to all edges in $N(A)$.   
  Note that there are $d(n-d)$ vertices of $N(A)$ since every neighbor $B$ of $A$ is specified uniquely by $B = (A \wo \set{a_B}) \cup \set{x_B}$ where $a_B \in A$ and $x_B\in A\compl$. Two vertices $B,C\in N(A)$ are adjacent iff $a_B=a_C$ or $x_B=x_C$. This implies that the graph induced on $N(A)$ is regular of degree $d-1+(n-d-1)=n-2$. Thus $A$ assigns a total weight of 
    \[
        w = d(n-d) + \frac12 \cdot \frac{ d(n-d)(n-2)}2 = d(n-d)\parens[\Big]{1 + \frac{n-2}4}
    \]
    to edges of the graph.
    
    We will now show that no edge of $\jnd$ receives a total weight of more than $1$ from this assignment. First, note that no edge is incident with two elements of $I$, for they would be adjacent. Similarly, if an edge is incident with $A\in I$ it cannot also be an edge in $N(A')$ for any $A'\in I$ for then we would have $A \adj A'$, a contradiction. Thus it only remains to prove that if $AB$ is an edge then there exist at most two elements $A'$ of $I$ that have $A,B\in N(A')$. 
    
    Let us consider what common neighbors of $A$ and $B$ look like. We know that $C=A\cap B$ has size $d-1$ and for some $x,y\in [n]$ we have $A = C\cup\set{x}$ and $B = C\cup \set{y}$. Consider now $A' \in N(A)\cap N(B)$. If $C\of A'$ then $A' = C\cup {z}$ for some $z\neq x,y$ in $C\compl$. We call such common neighbors \emph{type $1$}. Now if a neighbor $A'$ of $A$ is not of type $1$ then it has the form $(C\wo \set{c}) \cup \set{x,z}$ for some $c\in C$ and $z\not\in A$. But the only way such a set can also be a neighbor of $B$ is to have $z=y$. Thus all other common neighbors of $A$ and $B$ are \emph{type $2$} common neighbors: those of the form $(C\wo \set{c}) \cup\set{x,y}$. 
    
    Now we simply note that the type $1$ common neighbors of $A$ and $B$  are all pairwise adjacent to one-another in $\jnd$, as are the type $2$ common neighbors. That means at most one type 1 neighbor and at most one type 2 neighbor may be in $I$. Thus the edge $AB$ receives a weight of $1/2$ from at most one type $1$ common neighbor, and weight $1/2$ from at most one type $2$ common neighbor, for a total weight of at most $1$.
    
    Now we simply compute as follows. Each member of the independent set $I$ assigns total weight $w$ to the edges of $\jnd$, and each edge of $\jnd$ receives total weight at most $1$ from the elements of $I$, so
    \begin{align*}
        \abs{I}\, w = \abs{I}\, d(n-d)\parens[\Big]{1 + \frac{n-2}4} &\le \binom{n}d \frac{d(n-d)}2 = e(\jnd), \\
    \shortintertext{thus}
        \abs{I} \,\parens[\Big]{1 + \frac{n-2}4} &\le \binom{n}d \frac12 \\
        \abs{I} \,(n+2) &\le 2 \binom{n}d \\
        \abs{I} &\le \frac{2}{n+2} \binom{n}d. 
    \end{align*}
    \hfill $\square$
    }
    
    \section{Non-Negativity for Sparse Paving Matroids}\label{sec:positivity}
    With the formula for Theorem \ref{thm:main} proven, we now move to showing that this formula is always non-negative. When $\mcch$ is a disjoint family, this formula has a manifestly positive interpretation, as stated in the introduction of this paper. More details can be found in \cite{rhoremoved}. Otherwise, for more general cases of sparse paving matroids, we are not yet able to give a manifestly non-negative expression. Instead, we show directly that our formula from Theorem \ref{thm:main} is non-negative by relying on the bounds given in section \ref{sec:bounds} for $|\mcch|$,  our formulas for $\skyt(a,i,b)$ and $\bskyt(i,b)$ given in section \ref{sec:SKYT}, and some standard algebra and calculus tools. The details for this proof will be rather technical, and our proof will need a few cases, so the proof serves more as an outline, leaving most of the work to seperate Lemmas and Propositions. Throughout the proofs of this section, we use the falling factorial $(x)_{(n)}:=x(x-1)\cdots(x-n+1)$. We will also regularly use the fact $\deg P_M(t)<{1\over 2}\rk M$. That is, if $d$ is the rank of a matroid $M$, and $i$ is the power of some term in the Kazhdan-Lusztig polynomial $P_M(t)$, then we must have $i<d/2$. 

\begin{theorem} Let $\coeff$ be the Kazhdan-Lusztig coefficient for a sparse paving matroid $\smd$. Then 
\[\coeff\geq 0.\]
\end{theorem}
\begin{proof}
We are able to take care of most of the cases simultaneously. Since \[\coeff=\skyt(m+1,i,d-2i+1)-|\mcch|\cdot \bskyt(i,d-2i+1)\] by Theorem \ref{thm:combformula} and \(|\mcch|\leq {2\over m+d+2}{m+d\choose d}\) by Theorem \ref{thm:jamiesbound}, we have
\[\coeff\geq \skyt(m+1,i,d-2i+1)-{2\over m+d+2}{m+d\choose d}\cdot \bskyt(i,d-2i+1).\] 
Then by Lemma \ref{lem_pos_i_m_arb}, this expression is non-negative for $i\geq 3$, $m\geq 3$, and for all possible $d$. That is, for $d>2i$

This leaves a small number of more specific cases left, which need to be addressed independently. We first note that the cases for $m=0$ and $i=0$ are taken care of by Remark \ref{rem:conventions}.

When $m=1$, notice that any pair of basis elements have symmetric difference 2, and so $|\mcch|\leq 1$. In this case our desired result is immediate since by definition, we may view $\bSkyt(i,d-2i+1)$ as a subset of $\Skyt(2,i,d-2i+1)$.

When $m=2$, it is necessary to find a better bound on the size of $\mcch$. It is not too much work to show that $|\mcch|\leq {d+2\over 2}$ by using the symmetric difference condition on $\mcch$. It is easier to work with the complements of the elements in $\mcch$, which are elements of ${[d+2]\choose 2}$. Then it is equivalent in this case to count the size of the largest disjoint family in ${[d+2]\choose 2}$. So in the case of $m=2$ we have 

\[ \coeff\geq \skyt(m+1,i,d-2i+1)-{d+2\over 2}\cdot \bskyt(i,d-2i+1),\]
and so to prove our desired result in this case we need only prove 
\[\skyt(m+1,i,d-2i+1)-{d+2\over 2}\cdot \bskyt(i,d-2i+1)\geq 0.\]
We do this for $i\geq 1$, leaving the details to Lemma \ref{lem_m_2}.

Now we move on to the remaining values of $i$, noting we need only show them for $m\geq 3$. When $i=1$, one can get the following closed formula for $\skyt(m+1,i,d-2i+1)$. We get
\[\skyt(m+1,1,d-1)={m+d\choose d-1}-m-d\]
by Proposition \ref{prop_closedform_ieq1}. Also, note that $\bskyt(1,d-1)=d-1$, which can be seen by using Lemma \ref{lem:manifposbarskyt}, or by simply observing that only numbers in $\{2,3,4,\dots,d\}$ may appear below the position containing 1 in $\bskyt(1,d-1)$. It is also important to note that when $i=1$, $d\geq 3$. Then to get our desired result in this case, we can combine Theorem \ref{thm:combformula} and Theorem \ref{thm:codingbound} and instead show that
\[{m+d\choose d-1}-m-d-{1\over m+1}{m+d\choose d}(d-1)\geq 0.\]
Lemma \ref{lem_pos_i1} is able to show this for $d\geq 3$ when $m\geq 4$, but only for $d\geq 4$ when $m=3$. This leaves the case when $m=3$ and $d=3$ to be done explicitly. Note that \[\skyt(4,1,2)=9\]
and 
\[\bskyt(1,2)=2,\]
which can be easily verified by any of our formulas from section \ref{sec:SKYT}, or by hand. Then non-negativity follows from the fact that in the special case of $m=d=3$, we can guarantee $|\mcch|\leq 4$, which one verify via a constructive argument.

When $i=2$, we can use a similar strategy that we used for the $i\geq 3$ and $m\geq 3$ case described in Lemma \ref{lem_pos_i_m_arb}. However, there will be a bit more involved here, and so we leave the details of this final case to Lemma \ref{lem_pos_i2}.

\end{proof}

\begin{rem}
In the case of $m=d=3$, it is worth noting that finding the bound $|\mcch|\leq 4$ was necessary. Both bounds for $|\mcch|$ given by Theorem \ref{thm:codingbound} or Theorem \ref{thm:jamiesbound} give $|\mcch|\leq 5$, and $9-5\cdot 2=-1$. So in this special case, we needed to get a better bound on $|\mcch|$ than what either of our two bounds could provide. 
\end{rem}

\begin{lemma}\label{lem_pos_i_m_arb}
Let $i$ and $m$ both be at least 3. Then
\[ \skyt(m+1,i,d-2i+1)-{2\over m+d+2}{m+d\choose d}\bskyt(i,d-2i+1) \geq 0.\]
\end{lemma}
\begin{proof}

One can rewrite the sum in Lemma \ref{lem:manifposskyt} using Remark \ref{rem:rewriting}. After doing this, letting $a=m+1$ and $b=d-2i+1$, the $k=0$ term in the formula for $\skyt(m+1,i,d-2i+1)$ is
\begin{align*}
A:&={m+i-1\choose i}{m+d \choose d-i}{(d-i-2)_{(d-2i-1)}(m+d-i)_{(d-2i-1)}\over (d-2i-1)!(m+d-i)_{(d-2i)}}\\
&={(m+i-1)!(m+d)!(d-i-2)_{(d-2i-1)}(m+d-i)_{(d-2i-1)}\over i!(m-1)!(d-i)!(m+i)!(d-2i-1)!(m+d-i)_{(d-2i)}}\\
&={(m+d)!(d-i-2)_{(d-2i-1)}\over i!(m-1)!(d-i)!(m+i)(d-2i-1)!(m+i+1)}.
\end{align*}

Utilizing Lemma \ref{lem:manifposbarskyt}, we have 
\begin{align*}
B:&={2\over m+d+2}{m+d\choose d}\bskyt(i,d-2i+1)\\
&={4(m+d)!\over m!(m+d+2)(i+1)!(i-1)!(d-2i-1)!(d-i+1)(d-i-1)}.
\end{align*}

Note that \[\skyt(m+1,i,d-2i+1)-{2\over m+d+2}{m+d\choose d}\bskyt(i,d-2i+1)\geq A-B,\] so it suffices to show $A-B\geq 0$. Recall that $i< d/2$. Put another way, this says that $d-i>i>i-1$. Hence, we may combine $A-B$ in the following way.
\begin{align*}
A-B&=A{m(i+1)(m+d+2)(d-i+1)(d-i-1)\over m(i+1)(m+d+2)(d-i+1)(d-i-1)}-B{(d-i)_{(d-2i+1)}(m+i)(m+i+1)\over (d-i)_{(d-2i+1)}(m+i)(m+i+1)}\\
&={(m+d)!p(m,i,d)\over m!(m+d+2)(i+1)!(d-i)!(m+i)(m+i+1)(d-i+1)(d-i-1)(d-2i-1)!}
\end{align*}
where 
\[p(m,i,d)=(d-i-2)_{(d-2i-1)}m(i+1)(m+d+2)(d-i+1)(d-i-1)-4(d-i)_{(d-2i+1)}(m+i)(m+i+1).\]

Hence, it suffices to show that $p(m,i,d)\geq 0$. We can, in fact, reduce the problem further by simplifying $p(m,i,d)$. Observe that 
\[p(m,i,d)=(d-i-1)_{d-2i}[m(i+1)(m+d+2)(d-i+1)-4(m+i)(m+i+1)(d-i)],\]
so it now suffices to show
\[q(m,i,d):=m(i+1)(m+d+2)(d-i+1)-4(m+i)(m+i+1)(d-i)\geq 0.\]
We show this for $m,i\geq 3$ by viewing $q$ as a function of $m$. The desired result follows from the following three claims for $q$ as a function of $m$.
\begin{enumerate}
\item $q$ is quadratic and concave up;
\item the critical point of $q$ is negative; and
\item $q(m,i,d)\geq 0$ for $m=3$.
\end{enumerate}

Showing these are elementary exercises in algebra and calculus, so we just highlight the important parts.

For claim (1), note that the coefficient of $m^2$ in $q(m,i,d)$ is $(i+1)(d-i+1)-4(d-i)$, and that we assume $d>2i$ and $i\geq 3$. Hence this coefficient is non-negative.

For claim (2), it suffices to show the coefficient of $m$ in $q(m,i,d)$ is positive. This coefficient  is 
\[(i+1)(d+2)(d-i+1)-4(i+1)(d-i)-4i(d-i).\]
Using the fact that $d>2i$, one can show this is an increasing function in $d$ and is non-negative when $d=2i$.

For claim (3), it suffices to show $q(3,i,d)$ is an increasing function in $d$ and that $q(3,i,2i)$ is non-negative. This works out similarly to claim (2). 

\end{proof}

\begin{lemma}\label{lem_m_2}
Let $i\geq 1$ and $m=2$. Then

\[ \skyt(m+1,i,d-2i+1)-{d+2\over 2}\bskyt(i,d-2i+1) \geq 0.\]
\end{lemma}
\begin{proof}

As in Lemma \ref{lem_pos_i_m_arb}, keeping in mind that $m=2$, set
\begin{align*}
A:={(d+2)!(d-i-2)_{(d-2i-1)}\over i!(d-i)!(i+2)(d-2i-1)!(i+3)}
\end{align*}
and 
\begin{align*}
B:&={d+2\over 2}\bskyt(i,d-2i+1)\\
&={d!(d+2)\over (i+1)!(i-1)!(d-2i-1)!(d-i+1)(d-i-1)}.
\end{align*}
It follows from the proof of Lemma \ref{lem_pos_i_m_arb} that $\skyt(m+1,i,d-2i+1)\geq A$ for $m=2$, and so the desired result follows if we show $A-B\geq 0$.  Observe that

\begin{align*}
A-B=&A{(i+1)(d-i+1)(d-i-1)\over (i+1)(d-i+1)(d-i-1)}-B{(i+2)(i+3)(d-i)_{(d-2i+1)}\over (i+2)(i+3)(d-i)_{(d-2i+1)}}\\
&={ d!(d+2)p(i,d)\over (i+3)!(d-i)!(d-2i-1)! (d-i+1)(d-i-1)},
\end{align*}

where 
\[p(i,d):=(d-i-2)_{(d-2i-1)}(d+1)(i+1)(d-i+1)(d-i-1)-(i+2)(i+3)(d-i)_{(d-2i+1)}.\]
Hence, it suffices to show that $p(i,d)$ is non-negative. One can factor $p(i,d)$ to reduce the problem further: 
\[p(i,d)=(d-i-1)_{(d-2i)}[(d+1)(i+1)(d-i+1)-(i+2)(i+3)(d-i)],\]
and so it suffices to show that 
\[q(i,d):=(d+1)(i+1)(d-i+1)-(i+2)(i+3)(d-i)\]
is non-negative. Since in the context of Kazhdan-Lusztig polynomials we have $d>2i$, we may set $d=2i+j$ for $j\geq 1$. Then $q(i,2i+j)$ is quadratic in $j$ and we have the following values of $[j^\ell]q(i,2i+j)$:
\begin{align*}
[j^2]q(i,2i+j)&=i+1\\
[j^1]q(i,2i+j)&=2i^2-4\\
\text{Remaining terms}&\text{: }i^3-2i+1\\
\end{align*}
When $i\geq 2$, all three values are individually positive positive. If $i=1$, then 
\[q(1,j+2)=2j^2-2j\]
which is non-negative for all $j\geq 1$, giving our desired result.
\end{proof}

\begin{prop}\label{prop_closedform_ieq1}
\[\skyt(m+1,1,d-1)={m+d\choose d-1}-m-d.\]
\end{prop}
\begin{proof}
Note that if $\a\in \Skyt(m+1,1,d-1)$, it is made up of two ``tails'', one of length $m+1$ extending down, and the other of length $d-1$ extending up, so that the two tails overlap in exactly two positions. See the below figure for a schematic of $\a$, with some entries labeled. 

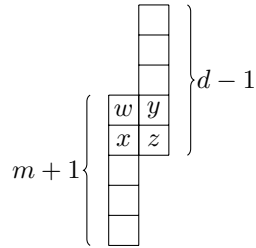
\begin{center}
\begin{tikzpicture}[scale=.4]
\draw (0,1) grid (-1,-4);
\draw[decoration={brace,raise=7pt},decorate]
 (-1,-4) -- node[left=7pt] {$m+1$} (-1,1);
\draw (0,-1) grid (1,4);
\draw[decoration={brace,raise=7pt},decorate]
 (1,4) -- node[right=7pt] {$d-1$} (1,-1);
 
 \node[] (x1) at (-.5,.5) {$w$};
 \node[] (x1) at (-.5,-.5) {$x$};
 \node[] (x1) at (.5,.5) {$y$};
 \node[] (x1) at (.5,-.5) {$z$};
\end{tikzpicture}
\end{center}

Note that there are $m+d$ positions in these tableaux, and we require that $w<y$ and $x<z$. Now, pick an element $\displaystyle S\in {[m+d]\choose d-1}$. The number of elements of $\Skyt(m+1,1,d-1)$ is equivalent to the number of $S$ that appear as the right tail in an element in $\Skyt(m+1,1,d-1)$, as the entries of one tail determine the entries of the other. It is easiest to count the complement, that is, the $S$ that will \textit{not} appear as the the right tail in an element of $\Skyt(m+1,1,d-1)$. These are the $S$ that force $w>y$, $x>z$, or both. We leave it to the reader to verify that the complement has size $m+d$.

\end{proof}
\begin{lemma}\label{lem_pos_i1}
We have \[{m+d\choose d-1}-m-d-{1\over m+1}{m+d\choose d}(d-1)\geq 0\]
for $d\geq 3$ when $m\geq 4$, and $d\geq 4$ when $m=3$.
\end{lemma}
\begin{proof}
We start by rewriting of our expression of interest. 
\begin{align*}
{m+d\choose d-1}-m-d-{1\over m+1}{m+d\choose d}(d-1)&={d\over m+1}{m+d\choose d}-m-d-{1\over m+1}{m+d\choose d}(d-1)\\
&={1\over m+1}{m+d\choose d} ({d}-(d-1))-m-d\\
&={1\over m+1} {m+d\choose d} -m-d\\
&={(m+d)_{(d-1)}\over d!}-m-d\\
&=(m+d)\left({(m+d-1)_{(d-2)}\over d!}-1\right).
\end{align*}
Hence, if\[f(m,d):={(m+d-1)_{(d-2)}\over d!}\] it suffices to show $f(m,d)\geq 1$. As a function in $m$, $f(m,d)$ is increasing. Also, \[f(4,d)={(d+3)_{(d-2)}\over d!}={(d+3)!\over 5!d!}={1\over 20}{d+3\choose 3}.\]
See that $f(4,d)$ is increasing in $d$ and also $f(4,3)=1$. So when $m\geq 4$, we have our desired result for $d\geq 3$. When $m=3$, observe we have 
\[f(3,d)={(d+2)_{(d-2)}\over d!}={(d+2)!\over 4!d!}={1\over 12}{d+2\choose 2}.\]
See that $f(3,d)$ is increasing in $d$, and $f(3,4)={15\over 12}$.
\end{proof}

\begin{lemma}\label{lem_pos_i2}
If $m\geq 3$, we have
\[c_{m,d}^2(\mcch)\geq 0.\]
\end{lemma}
\begin{proof}
It will be important to remember that since $i=2$, we have $d\geq 5$ by the degree requirement on Kazhdan-Lusztig polynomials. 

To show our desired result, we will need two separate cases. First suppose $m\geq d$. Note then we already have $m\geq 3$ since $d\geq 5$. As in Lemma \ref{lem_pos_i_m_arb}, accounting for the fact that in this case $i=2$, let 
\begin{align*}
A:&={(m+d)!(d-4)_{(d-5)}\over 2(m-1)!(d-2)!(m+2)(d-5)!(m+3)}\\
&={(m+d)!(d-4)!\over 2(m-1)!(d-2)!(m+2)(d-5)!(m+3)}\\
&={(m+d)!(d-4)m(m+1)\over 2(m+3)!(d-2)!}.
\end{align*}
Also similarly to Lemma \ref{lem_pos_i_m_arb}, but using the bound from Theorem \ref{thm:codingbound} for $|\mcch|$, let
\[B:={1\over m+1}{m+d\choose d}{2\cdot d!\over 6(d-5)!(d-1)(d-3)}={(m+d)!(d-2)(d-4)\over 3(m+1)!(d-1)!}.\]
A combination of Theorem \ref{thm:combformula}, Theorem \ref{thm:codingbound}, and the proof of Lemma \ref{lem_pos_i_m_arb} implies that 
\[c_{m,d}^2(\mcch)\geq A-B,\]
and so we show $A-B\geq 0$ when $m\geq d$. Notice that
\[A-B={(m+d)!(d-4)f(m,d) \over 6(m+3)!(d-1)!},\]
where \[f(m,d):=3m(m+1)(d-1)-2(d-2)(m+2)(m+3).\]
Hence, it suffices to show that $f(m,d)\geq 0$ to show that $A-B\geq 0$. Since we are assuming $m\geq d$, we set $m=d+j$, for $j\geq 0$. Then $f(d+j,d)$ is quadratic in $j$ and we have 
\begin{align*}
[j^2] f(d+j,d)&=d + 1\\
[j] f(d+j,d)&=2d^2 - 5d + 17\\
\text{Remaining terms of }f(d+j,d)&:d^3 - 6d^2 + 5d + 24
\end{align*}
Each of these are positive when $d=5$. In fact, the $[j^2]$ term is clearly positive when $d\geq 5$. The $[j]$ term is increasing for $d\geq {5\over 2}$. For the remaining terms, note that the derivative is $3d^2-12d+5$, which increases so long as $d\geq 2$, and is already positive at $d=5$. This means that the derivative remains positive for $d\geq 5$, and so the original function remains increasing. Hence, this shows that $c_{m,d}^2(\mcch)\geq 0$ so long as $m\geq d$.

Now we show the same result holds when $d\geq m$. To do this, we reuse $A$ as above, and redefine $B$ using our bound from Theorem \ref{thm:jamiesbound}. 
\[B:={2\over m+d+2}{m+d\choose d}{2\cdot d!\over 6(d-5)!(d-1)(d-3)}={2(m+d)!(d-2)(d-4)\over 3(m+d+2)m!(d-1)!}\]
For similar reasons as before, $c_{m,d}^2(\mcch)\geq 0$ if $A-B\geq 0$. Note that 
\[A-B={(m+d)!(d-4)(m+1)g(m,d) \over 6(m+3)!(d-1)!},\]
where 
\[g(m,d):=3m(m+d+2)(d-1)-4(d-2)(m+2)(m+3).\]
Observe that $g$ is a concave up quadratic function in $d$. If one expands the function, its vertex can be seen to occur at 
\[d={m^2+17m+24\over 6m}.\]

However, note that this value is less than $m$ so long as $m\geq 5$ since
\[{m^2+17m+24\over 6m}\leq m \text{ if and only if }-5m^2+17m+24\leq 0.\]
 Hence, this says that $g(m,d)$ is increasing in $d$ when $d\geq m\geq 5$. Also, when $m=3$ the vertex for $g$ is at approximately $d=4.67$ and when $m=4$ the vertex for $g$ is at $d=4.5$. We know that $d\geq 5$ regardless of its relation to $m$, so we have in fact shown that $g$ is increasing in $d$ for any $m\geq 3$ when $d\geq m$. Moreover, one can verify
\[g(m,m)=2(m^3 - 6m^2 + 5m + 24)\geq 0\]
so long as $m\geq 5$. Also, note that $g(3,5)=0$ and $g(4,5)=24$. Hence $g(m,d)$ is always non-negative for $d\geq m$ when $m\geq 3$.
\end{proof}

\section{Integral Identities}\label{sec:integral_identities}
	
\begin{prop}\cite[Identity 2.110.8]{gradshteyn}
Let $a,b$ be positive integers. Then
\[\int y^a(1-xy)^b \ dy=a!b!\sum_{k=0}^b{(1-xy)^{b-k}y^{a+k+1}x^k\over (a+k+1)!(b-k)!}.\]

\end{prop}

\begin{cor}\label{cor:int_0^1_y^a(1-xy)^b} Let $a,b$ be positive integers. Then
\[\int_0^1 y^a(1-xy)^b \ dy=a!b!\sum_{k=0}^b{(1-x)^{b-k}x^k\over (a+k+1)!(b-k)!}.\]
\end{cor}

\begin{cor}\label{cor:int_0^yx^a(1-x)^b} Let $a,b$ be positive integers. Then

\[\int_0^y x^a(1-x)^b\ dx=a!b!\sum_{k=0}^{b} {(1-y)^{b-k}y^{a+k+1}\over (a+k+1)!(b-k)!}\]
\end{cor}

\begin{prop}\label{prop:int_i_times}
Let $x_0,x_1,\dots, x_i$ be a list of $i+1$ variables. Set $h_1(x_1)=\displaystyle\int_0^{x_1} \ x_0^a(1-x_0)^b \ dx_0$, and for $i>1$ define $h_i(x_i)=\displaystyle\int_0^{x_i} \ h_{i-1}(x_{i-1}) \ dx_{i-1}$. Then \[\displaystyle\int_0^1\ h_i(x_i)\ dx_i ={a!(b+i)!\over i!(a+b+i+1)! }.\]
\end{prop}
\begin{proof}
Using Corollary \ref{cor:int_0^yx^a(1-x)^b} $i$ times, we get the following expression for $h_i(x_i)$.
\begin{align}\label{eq_sum}
h_i(x_i)=a!b!\sum_{k_1=0}^{b}\sum_{k_2=0}^{b-k_1}\ \sum_{k_3=0}^{b-k_1-k_2}\cdots \ \sum_{k_i=0}^{b-\sigma} {x_i^{a+\sigma+k_i+i}(1-x_i)^{b-\sigma-k_i}\over (a+\sigma+k_i+i)!(b-\sigma-k_i)!}
\end{align}
where $\sigma=k_1+k_2+\cdots +k_{i-1}$. Noting that 
\[\int_0^1 x_i^{a+\sigma+k_i+i}(1-x_i)^{b-\sigma-k_i}\ dx_i={(a+\sigma+k_i+i)!(b-\sigma-k_i)!\over (a+b+i+1)!}.\]
we may use \eqref{eq_sum} to write
\begin{align*}
\int_0^1 h_i(x_i)\ dx_i&=a!b!\sum_{k_1=0}^{b}\ \ \sum_{k_2=0}^{b-k_1}\ \ \sum_{k_3=0}^{b-k_1-k_2}\ \cdots \ \sum_{k_i=0}^{b-\sigma} {(a+\sigma+k_i+i)!(b-\sigma-k_i)!\over (a+\sigma+k_i+i)!(b-\sigma-k_i)!(a+b+i+1)!}\\
&={a!b!\over (a+b+i+1)!}\sum_{k_1=0}^{b}\ \ \sum_{k_2=0}^{b-k_1}\ \ \sum_{k_3=0}^{b-k_1-k_2}\ \cdots \ \sum_{k_i=0}^{b-\sigma} 1,
\end{align*}

which simplifies using Proposition \ref{prop_combident} to 
\[{a!b!\over (a+b+i+1)!}{b+i\choose i}={a!(b+i)!\over i!(a+b+i+1)! }.\]
\end{proof}

\begin{prop}\label{prop_combident}
\[\sum_{k_1=0}^{b}\ \sum_{k_2=0}^{b-k_1}\ \sum_{k_3=0}^{b-k_1-k_2}\cdots \sum_{k_i=0}^{b-\sigma}1={b+i\choose i},\]
where $\sigma=k_1+k_2+\cdots +k_{i-1}$.
\end{prop}
\begin{proof}
It is helpful to first reindex the summations so that they start at 1 instead of 0. Then the identity holds from counting the below set in two ways. 
\[\bigcup_{x_1\in [b+1]}\ \bigcup_{x_2\in [b+2]\setminus [x_1]}\ \bigcup_{x_3\in [b+3]\setminus [x_2]}\cdots \bigcup_{x_i\in [b+i]\setminus [x_{i-1}]}\{x_1,x_2,\dots, x_i\}.\]

\end{proof}

\end{document}